\documentclass[bj,authoryear]{imsart}

\usepackage{amsmath,amssymb,amsthm}
\usepackage{booktabs}
\usepackage{graphicx}

\makeatletter\def\journal@name{Preprint submitted to Electronic Journal of Statistics}\makeatother
\startlocaldefs
\numberwithin{equation}{section}
\theoremstyle{plain}
\newtheorem{theorem}{Theorem}[section]
\newtheorem{proposition}[theorem]{Proposition}
\newtheorem{lemma}[theorem]{Lemma}
\newtheorem{corollary}[theorem]{Corollary}

\theoremstyle{definition}
\newtheorem{definition}[theorem]{Definition}

\newtheorem{remark}[theorem]{Remark}
\newtheorem{example}[theorem]{Example}

\newcommand{\E}{\mathbb{E}}
\newcommand{\PP}{\mathbb{P}}
\newcommand{\R}{\mathbb{R}}
\newcommand{\dd}{\mathrm{d}}
\newcommand{\KL}{\mathrm{KL}}
\newcommand{\TV}{\mathrm{TV}}
\newcommand{\Cg}{\mathcal{C}_\gamma}
\newcommand{\CH}{\mathcal{C}_H}
\newcommand{\lb}{\bar\lambda}
\newcommand{\parhead}[1]{\par\medskip\noindent\textbf{#1}\enspace}
\endlocaldefs

\begin{document}

\begin{frontmatter}
\title{Minimax and adaptive inference for branching ratios under long memory}
\runtitle{Branching-ratio inference under long memory}

\begin{aug}
\author[A]{\inits{M.}\fnms{Mauricio}~\snm{Herrera-Mar\'in}\ead[label=e1]{mherrera@udd.cl}\orcid{0000-0002-9604-3077}}
\address[A]{Faculty of Engineering, Universidad del Desarrollo, Avda. Plaza 680, Las Condes, Santiago, Chile\printead[presep={,\ }]{e1}}
\end{aug}

\begin{abstract}
The branching ratio $\rho$ of a self-exciting point process determines the branching margin $1-\rho$ to the critical boundary. We study inference on $\rho$ from a finite record for stationary Hawkes processes with completely monotone kernels whose slow relaxation mass obeys an envelope near zero rate. Slow excitation can then be exchanged for exogenous immigration; a spectral bound on the Kullback--Leibler rate and a coupling argument for the unobserved past give a minimax lower bound of order $T^{-\gamma/(2\gamma+1)}$ for a power-law envelope with tail exponent $\gamma$. A block-dispersion estimator whose block length grows as $T^{1/(2\gamma+1)}$ attains this rate. A quantitative central limit theorem, obtained from the cluster representation and a second-order Poincar\'e inequality on Poisson space, yields honest confidence intervals of nearly minimax expected length. For unknown $\gamma$, a Lepski-type rule adapts in probability at a logarithmic cost; free adaptation between memory classes is impossible, and honest intervals cannot adapt to lighter memory. Simulations are consistent with the rates.
\end{abstract}


\end{frontmatter}

\section{Introduction}\label{sec:intro}

Self-exciting point processes of Hawkes type \citep{Hawkes1971,HawkesOakes1974} are used to model clustered events in seismology, finance, epidemiology and social media. In the linear model the conditional intensity is
\begin{equation}\label{eq:hawkes}
\lambda(t)=\mu+\int_{(-\infty,t)}\phi(t-s)\,N(\dd s),
\end{equation}
with baseline $\mu>0$ and excitation kernel $\phi\ge0$. The \emph{branching ratio} $\rho=\int_0^\infty\phi(t)\,\dd t$ is the mean number of events directly triggered by one event. The process has a stationary version if and only if $\rho<1$, and $1-\rho$ is its \emph{branching margin} to the critical boundary. In finance, $\rho$ is read as the fraction of activity generated endogenously, and whether markets operate close to $\rho=1$ has been debated for more than a decade \citep{FilimonovSornette2015,HardimanBercotBouchaud2013,WheatleyWehrliSornette2019,WehrliWheatleySornette2021}. Part of that debate concerns biases produced by non-stationary immigration; another part concerns the choice between exponential and power-law kernels.

This paper isolates a third, more fundamental source of difficulty: \emph{slow excitation is statistically confounded with immigration}. A kernel with long memory spreads part of its mass over time scales comparable to, or longer than, the observation window. On such scales a triggered event and an exogenous arrival leave essentially the same trace. A finite record can therefore pin down the dynamics on observable scales far better than it pins down the zero-frequency mass $\rho=\widehat\phi(0)$, which is exactly the quantity that sets the inferred branching margin.

We make this precise for kernels that are completely monotone. By Bernstein's theorem \citep{SchillingSongVondracek2012} such a kernel is a mixture of exponential relaxations,
\begin{equation}\label{eq:bernstein}
\phi(t)=\int_{(0,\infty)} r\,e^{-rt}\,\nu(\dd r),\qquad \rho=\nu((0,\infty)),
\end{equation}
where the \emph{branching--relaxation measure} $\nu$ assigns branching mass to relaxation rates (the Bernstein measure of $\phi$ in the usual normalisation is $r\,\nu(\dd r)$). Long memory corresponds to mass of $\nu$ near $r=0$. We consider classes $\CH$ in which the slow relaxation mass obeys an envelope, $\nu((0,r])\le c_+H(r)$ for small $r$ (Definition~\ref{def:class}); the power-law envelope $H(r)=r^\gamma$ contains Omori-type kernels $\phi(t)\propto(1+t)^{-1-\gamma}$.

\parhead{Main result.} Theorem~\ref{thm:main} shows that the difficulty of estimating $\rho$ is governed by the envelope: if $r_T$ solves $T\,r_T\,H(r_T)^2\asymp1$, then
\[
\inf_{\hat\rho}\sup_{\CH}\E|\hat\rho-\rho|\ \gtrsim\ H(r_T),
\]
where the infimum is over all estimators based on a stationary record of length $T$ with unobserved past. For power-law envelopes this gives the rate $T^{-\gamma/(2\gamma+1)}$ (Corollary~\ref{cor:power}), and for every envelope $\rho$ fails to be uniformly $\sqrt T$-estimable (Corollary~\ref{cor:roott}). For power-law envelopes the rate is attained by a block-dispersion estimator whose block length grows as $T^{1/(2\gamma+1)}$ (Theorem~\ref{thm:upper}), so $T^{-\gamma/(2\gamma+1)}$ is the minimax rate. With a fixed block length the estimator retains a non-vanishing finite-window bias; under long memory this bias decays only as $W^{-\gamma}$, which dictates how fast the block length must grow. Our upper bound thus turns the dispersion approximation of \citet{HardimanBouchaud2014} into a rate-optimal estimator by letting the counting scale grow with the observation horizon. For inference, Theorem~\ref{thm:clt} gives a Wasserstein bound of order $\sqrt{W/T}$ for the normal approximation of the block-dispersion statistic, uniformly over the class, and Corollary~\ref{cor:ci} and Proposition~\ref{prop:cilength} show that honest confidence intervals exist and have minimax length up to an arbitrarily slowly growing factor. Finally, for unknown $\gamma$, a Lepski-type estimator attains the rate up to a logarithmic factor in probability (Theorem~\ref{thm:adapt}); any estimator that is rate-optimal for lighter memory loses a logarithmic factor for heavier memory (Theorem~\ref{thm:price}); and honest confidence intervals over a heavier-memory class cannot shrink at the faster rate of a lighter-memory class (Proposition~\ref{prop:noadaptci}). The proof of Theorem~\ref{thm:main} is a two-point argument. The alternative is obtained from the null by removing the Bernstein mass below a rate $r_0$ and compensating it with extra immigration, so that the mean intensity is unchanged while $\rho$ decreases by the removed mass. Two ingredients control the divergence between the two laws on a window:
\begin{itemize}
\item a spectral bound on the Kullback--Leibler rate, obtained from the Bartlett spectrum and Plancherel's identity (see Proposition~\ref{prop:rate});
\item a coupling lemma showing that the unobserved past contributes only a vanishing boundary term, because the alternative has no relaxation rates below $r_0$ and hence forgets its history exponentially fast (Lemma~\ref{lem:coupling}).
\end{itemize}

\parhead{From approximation to an inverse problem.} Section~\ref{sec:pseudo} explains how the lower bound manifests itself for a common practice: fitting kernels that are sums of $K$ exponentials. For small $K$ the fitted branching ratio converges to a pseudo-true value determined by a weighted spectral approximation problem, and the bias is an approximation error. For larger $K$ the criterion develops an almost flat direction in which slow components exchange mass with the baseline, and the remaining bias is a resolution limit of the kind quantified by Theorem~\ref{thm:main}. Better approximation of the kernel therefore does not imply better identification of $\rho$.

\parhead{Relation to the literature.} Nonparametric estimation of Hawkes kernels is well developed. \citet{ReynaudBouretSchbath2010} obtain adaptive minimax estimators for H\"older kernels with unknown support, \citet{HansenReynaudBouretRivoirard2015} study Lasso estimators, \citet{DonnetRivoirardRousseau2020} derive posterior concentration rates in $L^1$ norms of the interaction functions, and \citet{SulemRivoirardRousseau2024} extend Bayesian estimation to nonlinear Hawkes processes; \citet{BacryJaissonMuzy2016} estimate slowly decreasing kernels nonparametrically from high-frequency data. These works target the recovery of the kernel in functional norms, typically on controlled memory horizons. Our target is different: the scalar zero-frequency functional $\rho=\widehat\phi(0)$ when part of the branching mass may sit in arbitrarily slow relaxation modes that the data confound with immigration. The difficulty is one of extrapolation to zero frequency, and it makes $\rho$ an irregular functional even though the observable dynamics can be estimated well.

Scaling limits of nearly unstable Hawkes processes are well understood \citep{JaissonRosenbaum2015,JaissonRosenbaum2016,Xu2024}; they describe \emph{which} limiting object governs near-critical dynamics, but not what can be recovered from finite data. Asymptotic theory for maximum likelihood exists for exponential and, more generally, short-range kernels \citep{Ogata1978,ClinetYoshida2017}. Recent work on inference for the branching ratio under non-stationary baselines assumes the short-range condition $\int_0^\infty t\,\phi(t)\,\dd t<\infty$ and leaves the heavy-tailed case with infinite first moment open \citep{PotironScailletVolkovYu2025}. Direct estimators of the branching ratio exist: \citet{HardimanBouchaud2014} use the mean and variance of counts in large windows, and \citet{AchabBacryMuzyRambaldi2018} estimate branching-ratio matrices nonparametrically from integrated cumulants. These works do not characterise the minimax difficulty created by arbitrarily slow relaxation mass. The balance between block length and number of blocks in our upper bound is reminiscent of aggregated-variance methods for long-memory time series \citep{Beran1994}, but here the long-memory parameter enters through the bias of a short-memory quantity, $f(0)$, rather than through the target itself. Quantitative central limit theorems for linear functionals of Hawkes processes have been obtained with the Malliavin--Stein method through Poisson embeddings \citep{Torrisi2016,HillairetHuangKhabouReveillac2022,KhabouPrivaultReveillac2024}. Our statistic is quadratic in the counts, and we work instead with the cluster representation, which turns the Hawkes process into a marked Poisson process with independent marks; the second-order Poincar\'e inequality of \citet{LastPeccatiSchulte2016} then yields bounds that are uniform over a long-memory class. We are not aware of a minimax analysis of the branching ratio under long memory.

\parhead{Organisation.} Section~\ref{sec:setting} introduces the model and the class. Section~\ref{sec:main} states the minimax results, the central limit theorem, the confidence intervals and the adaptation results, proved in Section~\ref{sec:proofs}. Section~\ref{sec:pseudo} discusses rational pseudo-truths. Section~\ref{sec:numerics} reports numerical experiments, and Section~\ref{sec:discussion} discusses structured estimation, the near-critical regime and open problems.

\section{Setting}\label{sec:setting}

\subsection{Stationary Hawkes processes with completely monotone kernels}

Let $N$ be a simple point process on $\R$ with intensity~\eqref{eq:hawkes} with respect to its internal history, where $\phi$ has the representation~\eqref{eq:bernstein} for a finite measure $\nu$ on $(0,\infty)$ with $\rho=\nu((0,\infty))<1$ and $\phi(0)=\int r\,\nu(\dd r)<\infty$. Then $N$ admits a unique stationary version \citep{BremaudMassoulie1996}, with mean intensity $\lb=\mu/(1-\rho)$ and Bartlett spectral density
\begin{equation}\label{eq:bartlett}
f(\omega)=\frac{\lb}{|1-\widehat\phi(\omega)|^2},\qquad \widehat\phi(\omega)=\int_0^\infty e^{-i\omega t}\phi(t)\,\dd t=\int\frac{r}{r+i\omega}\,\nu(\dd r),
\end{equation}
see \citet{Hawkes1971} and \citet{DaleyVereJones2003}. In particular, for $\psi\in L^1\cap L^2(\R_+)$,
\begin{equation}\label{eq:varformula}
\mathrm{Var}\Big(\int_{(-\infty,t)}\psi(t-s)\,N(\dd s)\Big)=\frac{1}{2\pi}\int_\R|\widehat\psi(\omega)|^2f(\omega)\,\dd\omega .
\end{equation}
We use the cluster representation of \citet{HawkesOakes1974}: $N$ is the superposition of independent clusters, each started by an immigrant from a Poisson process of rate $\mu$ and consisting of the immigrant and its descendants. The expected density of descendants at lag $u$ is the resolvent $\Psi=\sum_{n\ge1}\phi^{*n}$.

\subsection{The class}

\begin{definition}\label{def:class}
Let $0<r_\star\le1$ and let $H:(0,r_\star)\to(0,\infty)$ be continuous and non-decreasing with $H(r)\to0$ as $r\downarrow0$, and fix $c_+>0$, $\rho_{\max}<1$, $0<\mu_{\min}\le\mu_{\max}$ and $A<\infty$. The class $\CH$ consists of all pairs $(\mu,\nu)$ with $\mu\in[\mu_{\min},\mu_{\max}]$, $\nu((0,\infty))\le\rho_{\max}$, $\int r\,\nu(\dd r)\le A$ and
\begin{equation}\label{eq:tailcond}
\nu((0,r])\le c_+\,H(r)\qquad(0<r<r_\star).
\end{equation}
For $\gamma\in(0,1)$ we write $\Cg$ for the power-law case $H(r)=r^\gamma$. The restriction $r_\star\le1$ only concerns the behaviour near zero rate; it makes the power-law classes nested, $\mathcal C_{\gamma_2}\subset\mathcal C_{\gamma_1}$ for $\gamma_1<\gamma_2$, when their other constants coincide.
\end{definition}

Complete monotonicity supplies the Bernstein representation; condition~\eqref{eq:tailcond}, an upper envelope for the slow relaxation mass, is what determines the rate. The class imposes only this upper envelope; a matching lower regularity is required only of the least-favourable point used in the lower-bound construction (condition~\eqref{eq:nulltail} below). An upper envelope is also the only tail information that a truncation-based estimator would use. Note that $H$ describes the mass of the relaxation spectrum near zero rate, that is, how much branching is carried by very slow modes; it is not a smoothness index of the kernel in time.

\begin{example}\label{ex:omori}
The Omori kernel, $\phi(t)=\rho\,\gamma(1+t)^{-1-\gamma}$, satisfies the representation~\eqref{eq:bernstein} with $\nu(\dd r)=\rho\,\Gamma(\gamma)^{-1}r^{\gamma-1}e^{-r}\,\dd r$. Hence $\nu((0,r])\sim\rho\,r^\gamma/\Gamma(\gamma+1)$ as $r\to0$ and $(\mu,\nu)\in\Cg$ for suitable constants. Envelopes such as $H(r)=r^\gamma\log(1/r)^\beta$ describe logarithmic corrections of the same tail. Its tail $\int_t^\infty\phi\sim\rho\,t^{-\gamma}$ has infinite first moment.
\end{example}

\subsection{Observation scheme and risk}

We observe the stationary process on $[0,T]$ and write $P_{\mu,\nu}^T$ for the law of $N|_{[0,T]}$. For an estimator $\hat\rho=\hat\rho(N|_{[0,T]})$ the minimax risk is
\[
\mathcal R_T(\CH)=\inf_{\hat\rho}\ \sup_{(\mu,\nu)\in\CH}\E_{\mu,\nu}\big|\hat\rho-\rho(\nu)\big| .
\]
The history before time $0$ is not observed. This is the situation of any empirical window, and it is the main technical point in the lower bound.

\section{Main result}\label{sec:main}

\begin{theorem}\label{thm:main}
Assume that $\CH$ contains a pair $(\mu_0,\nu_0)$ with $\mu_0<\mu_{\max}$, $\rho_0=\nu_0((0,\infty))<1$ and, for some $0<c_0\le C_0\le c_+$,
\begin{equation}\label{eq:nulltail}
c_0\,H(r)\le\nu_0((0,r])\le C_0\,H(r)\qquad(0<r<r_\star).
\end{equation}
Let $\kappa_0=(1-\rho_0)^3/(4C_0^2)$ and let $r_T$ be the solution of
\begin{equation}\label{eq:rT}
T\,r_T\,H(r_T)^2=\kappa_0 .
\end{equation}
If $H(r_T)^2\log T\to0$, then
\begin{equation}\label{eq:mainbound}
\liminf_{T\to\infty}\,\frac{\mathcal R_T(\CH)}{H(r_T)}\ \ge\ \frac{c_0}{8} .
\end{equation}
\end{theorem}

The function $r\mapsto rH(r)^2$ is continuous and strictly increasing, so $r_T$ is well defined for $T$ large and $r_T\downarrow0$. Equation~\eqref{eq:rT} balances the information carried by slow relaxation modes against the length of the record, and $H(r_T)$ is the corresponding modulus of identifiability of $\rho$.

\begin{corollary}[Power-law envelope]\label{cor:power}
For $H(r)=r^\gamma$, $\gamma\in(0,1)$, one has $r_T=(\kappa_0/T)^{1/(2\gamma+1)}$ and
\[
\liminf_{T\to\infty}\,T^{\gamma/(2\gamma+1)}\,\mathcal R_T(\Cg)\ \ge\ \frac{c_0}{8}\Big[\frac{(1-\rho_0)^3}{4C_0^2}\Big]^{\gamma/(2\gamma+1)} .
\]
\end{corollary}

\begin{corollary}[No uniform root-$T$ rate]\label{cor:roott}
Under the assumptions of Theorem~\ref{thm:main}, $\sqrt T\,\mathcal R_T(\CH)\to\infty$. The branching ratio is an irregular functional over $\CH$: it is not uniformly $\sqrt T$-estimable.
\end{corollary}

\begin{proof}
By~\eqref{eq:rT}, $\sqrt T\,H(r_T)=\sqrt{\kappa_0/r_T}\to\infty$, and~\eqref{eq:mainbound} applies. \qedhere
\end{proof}

\parhead{A rate-optimal estimator.} The lower bound is attained by a block-dispersion estimator. Split $[0,T]$ into $M=\lfloor T/W\rfloor$ blocks $B_j$ of length $W$, let $T'=MW$, $N_j=N(B_j)$, and
\begin{equation}\label{eq:blockest}
\hat\lambda=\frac{N((0,T'])}{T'},\qquad \hat V=\frac1{T'}\sum_{j=1}^M\big(N_j-W\hat\lambda\big)^2,\qquad \hat\rho=\Pi_{[0,\rho_{\max}]}\Big(1-\sqrt{\tilde\lambda/\tilde V}\Big),
\end{equation}
where $\tilde\lambda$ and $\tilde V$ are $\hat\lambda$ and $\hat V$ clipped to the intervals $[\mu_{\min},\lambda_{\max}]$ and $[\mu_{\min},v_{\max}]$, with $\lambda_{\max}=\mu_{\max}/(1-\rho_{\max})$ and $v_{\max}=\mu_{\max}/(1-\rho_{\max})^3$, and $\Pi_{[0,\rho_{\max}]}$ is the projection onto $[0,\rho_{\max}]$. The estimator uses $f(0)=\lb/(1-\rho)^2$: $\hat\lambda$ estimates $\lb$ and $\hat V$ estimates $f(0)$. It is the mean--variance idea of \citet{HardimanBouchaud2014}, with a block length that grows with $T$.

\begin{theorem}\label{thm:upper}
Let $\gamma\in(0,1)$ and $W_T=\lceil T^{1/(2\gamma+1)}\rceil$. There is a constant $C$, depending only on $\gamma$ and the constants of $\Cg$, such that for all $T\ge2W_T$
\[
\sup_{(\mu,\nu)\in\Cg}\E_{\mu,\nu}|\hat\rho-\rho(\nu)|\ \le\ C\,T^{-\gamma/(2\gamma+1)} .
\]
\end{theorem}

\begin{corollary}[Minimax rate]\label{cor:minimax}
If $\Cg$ contains a point satisfying~\eqref{eq:nulltail} with $H(r)=r^\gamma$, then $\mathcal R_T(\Cg)\asymp T^{-\gamma/(2\gamma+1)}$.
\end{corollary}

The proof (Section~\ref{sec:proofupper}) separates a bias of order $W^{-\gamma}$, coming from the spectral behaviour of $f$ near zero frequency, from a stochastic error of order $\sqrt{W/T}$, controlled by the cluster representation. A fixed block length leaves a non-vanishing finite-window bias, since a block of length $W$ estimates a smoothed version of $f(0)$; under long memory this bias decays only as $W^{-\gamma}$.

\begin{remark}[General envelopes]
The same bias--variance argument suggests an extension to envelopes that are regularly varying at $0$ with index in $(0,1)$: with $W_T=1/r_T$, Potter's bounds would control the spectral bias by a multiple of $H(1/W_T)$, while the stochastic error is of order $\sqrt{W_T/T}\asymp H(r_T)$ by~\eqref{eq:rT}. We do not pursue this here.
\end{remark}

\begin{remark}[Rate]
In regular short-memory parametric Hawkes models, branching parameters typically admit the usual $T^{-1/2}$ rate under standard ergodicity and regularity conditions \citep{Ogata1978,ClinetYoshida2017}. Corollary~\ref{cor:roott} shows that no such uniform rate is available over long-memory classes, and Corollary~\ref{cor:minimax} identifies the exact rate for power-law envelopes. For power-law envelopes the exponent $\gamma/(2\gamma+1)$ increases from $0$ to $1/3$ as $\gamma$ goes from $0$ to $1$: heavier memory makes the branching ratio harder to estimate. The formal resemblance with nonparametric rates for $\gamma$-smooth functions is misleading, since $\gamma$ here indexes spectral mass near zero rate, not temporal regularity.
\end{remark}

\begin{remark}[Where the difficulty lies]
The alternative in the proof has the same mean intensity as the null and differs only by slow mass moved into the baseline. The relevant resolution scale is $\tau^*=1/r_T$, which for power-law envelopes is of order $T^{1/(2\gamma+1)}$, much shorter than $T$. Branching mass carried by relaxation modes slower than $\tau^*$ cannot be reliably separated from immigration, even though their characteristic times are well inside the window.
\end{remark}

\parhead{Asymptotic normality and confidence intervals.} We now quantify the uncertainty of the block-dispersion estimator. Let
\[
\check\rho=1-\sqrt{\hat\lambda/\hat V}
\]
be the unprojected version of~\eqref{eq:blockest}, defined on the event $\{\hat V>0\}$, and write $\sigma_W^2=\mathrm{Var}\big(\sum_{j=1}^MY_j^2\big)$ with $Y_j=N_j-W\lb$.

\begin{theorem}[Quantitative central limit theorem]\label{thm:clt}
There is a constant $C$, depending only on the constants of $\Cg$, such that for all $(\mu,\nu)\in\Cg$ and all $1\le W\le T/2$,
\[
d_{\mathrm W}\Big(\frac{\sum_jY_j^2-\E\sum_jY_j^2}{\sigma_W},\,\mathcal N(0,1)\Big)\ \le\ C\Big(\sqrt{\frac WT}+\frac1{\sqrt T}\Big),
\]
where $d_{\mathrm W}$ is the Wasserstein distance, and
\[
\Big|\frac{\sigma_W^2}{2\,T'W\,f(0)^2}-1\Big|\ \le\ C\big(W^{-\gamma}+W^{-1}\big).
\]
\end{theorem}

\begin{corollary}[Uniform asymptotic normality]\label{cor:an}
Let $W_T\to\infty$ with $W_T/T\to0$ and $T\,W_T^{-(2\gamma+1)}\to0$. Then
\[
\sup_{(\mu,\nu)\in\Cg}\ \sup_{x\in\R}\Big|\PP_{\mu,\nu}\Big(\sqrt{\tfrac{2T'}{W_T}}\,\frac{\check\rho-\rho}{1-\check\rho}\le x\Big)-\Phi(x)\Big|\ \longrightarrow\ 0 .
\]
\end{corollary}

The condition $TW_T^{-(2\gamma+1)}\to0$ makes the block length slightly longer than the rate-optimal choice, so that the bias $O(W_T^{-\gamma})$ is negligible against the standard deviation, which is of order $\sqrt{W_T/T}$. This is the usual price of bias-free inference in nonparametric problems.

\begin{corollary}[Honest confidence intervals]\label{cor:ci}
Let $\ell_T\to\infty$ with $\ell_T=o(T^{2\gamma/(2\gamma+1)})$, $W_T=\lceil T^{1/(2\gamma+1)}\ell_T\rceil$ and, for $\alpha\in(0,1)$,
\[
I_T=\begin{cases}\Big[\check\rho\pm z_{1-\alpha/2}\,\hat s\,\sqrt{W_T/(2T')}\,\Big]&\text{on }\mathcal R=\{\hat V\ge\mu_{\min}/2,\ \hat\lambda\ge\mu_{\min}/2\},\\[2pt] [0,\rho_{\max}]&\text{on }\mathcal R^c,\end{cases}
\]
where $\hat s=\Pi_{[1-\rho_{\max},1]}\big(\sqrt{\hat\lambda/\hat V}\big)$ is the clipped estimate of $1-\rho$. Then $I_T$ is defined for every sample, $\inf_{(\mu,\nu)\in\Cg}\PP_{\mu,\nu}(\rho\in I_T)\to1-\alpha$, and
\[
\sup_{(\mu,\nu)\in\Cg}\E_{\mu,\nu}|I_T|\ \le\ C\,T^{-\gamma/(2\gamma+1)}\ell_T^{1/2}.
\]
\end{corollary}

\begin{proposition}[Length of honest intervals]\label{prop:cilength}
Under the assumptions of Corollary~\ref{cor:power}, let $\alpha<(1-8^{-1/2})/2$ and let $I_T$ be any interval, measurable with respect to $N|_{[0,T]}$, with $\inf_{\Cg}\PP(\rho\in I_T)\ge1-\alpha$. Then $\sup_{\Cg}\E|I_T|\ge c\,T^{-\gamma/(2\gamma+1)}$ for $T$ large, with $c>0$ depending only on $\alpha$ and the least-favourable point.
\end{proposition}

Corollary~\ref{cor:ci} and Proposition~\ref{prop:cilength} show that honest confidence intervals exist and that their length is minimax up to the factor $\ell_T^{1/2}$, which can grow arbitrarily slowly.

\parhead{Adaptation to unknown memory.} The block length of Theorem~\ref{thm:upper} depends on $\gamma$. Fix $0<\gamma_-\le\gamma_+<1$ and consider the classes $\Cg$, $\gamma\in[\gamma_-,\gamma_+]$, with common constants. Let $W_k=2^k$ for $k_0\le k\le K_T$, with $W_{K_T}\le T/16$, and let $\hat\lambda_k=N((0,T'_k])/T'_k$ and $\hat V_k$ be the statistics~\eqref{eq:blockest} with block length $W_k$ and $T'_k=\lfloor T/W_k\rfloor W_k$. By Lemma~\ref{lem:var} there is a constant $C_V$, depending only on the class, with $\mathrm{Var}(\hat V_k)\le C_VW_k/T$; put $\bar\sigma_k=\sqrt{C_VW_k/T}$. For $x_T>0$ define
\begin{equation}\label{eq:lepski}
\hat k=\min\Big\{k:\ |\hat V_k-\hat V_l|\le6\,x_T\,\bar\sigma_l\ \text{for all }l>k\Big\},\qquad \hat\rho_{\rm ad}=\Pi_{[0,\rho_{\max}]}\Big(1-\sqrt{\tilde\lambda_{\hat k}/\tilde V_{\hat k}}\Big),
\end{equation}
where $\tilde\lambda_{\hat k}$ and $\tilde V_{\hat k}$ are clipped as in~\eqref{eq:blockest}. The rule does not use $\gamma$.

\begin{theorem}[Adaptive estimation]\label{thm:adapt}
Let $\varepsilon>0$ and $x_T^2=(\log T)^{1+\varepsilon}$. There is a constant $C$ such that
\[
\lim_{T\to\infty}\ \sup_{\gamma\in[\gamma_-,\gamma_+]}\ \sup_{(\mu,\nu)\in\Cg}\PP_{\mu,\nu}\Big(|\hat\rho_{\rm ad}-\rho|>C\Big(\frac{(\log T)^{1+\varepsilon}}{T}\Big)^{\gamma/(2\gamma+1)}\Big)=0 .
\]
\end{theorem}

The next result shows that free adaptation cannot hold: an estimator that is rate-optimal on a lighter-memory class must incur at least a logarithmic penalty on a heavier-memory class.

\begin{theorem}[Price of adaptation]\label{thm:price}
Let $0<\gamma_1<\gamma_2<1$. Assume that there are measures $\nu^{(1)},\nu^{(2)}$ and a baseline $\mu_s$ such that $c_0r^{\gamma_1}\le\nu^{(1)}((0,r])\le C_0r^{\gamma_1}$ for $r<r_\star$, $(\mu_s,\nu^{(2)}|_{(r,\infty)})\in\mathcal C_{\gamma_2}$ for all small $r$, $(\mu_s-\lb m,\,\nu^{(2)}|_{(r,\infty)}+\nu^{(1)}|_{(0,r]})\in\mathcal C_{\gamma_1}$ for all small $r$, where $m=\nu^{(1)}((0,r])$ and $\lb$ is the mean intensity of $(\mu_s,\nu^{(2)}|_{(r,\infty)})$, and $\nu^{(1)}((0,\infty))+\nu^{(2)}((0,\infty))<\rho_{\max}$. If an estimator satisfies $\sup_{\mathcal C_{\gamma_2}}\E|\hat\rho-\rho|\le C\,T^{-\gamma_2/(2\gamma_2+1)}$ for all $T$, then
\[
\liminf_{T\to\infty}\Big(\frac{T}{\log T}\Big)^{\gamma_1/(2\gamma_1+1)}\sup_{\mathcal C_{\gamma_1}}\E|\hat\rho-\rho|\ >\ 0 .
\]
\end{theorem}

The two results are of different nature and should be read together with care. Theorem~\ref{thm:adapt} is a logarithmic adaptive bound in probability; its failure probability is only of order $(\log T)^{-\varepsilon}$, so it does not yield a bound on the expected loss. Theorem~\ref{thm:price} is a statement about any estimator that is rate-optimal over the class with lighter memory: its proof shows that such an estimator misses $\rho$ by at least $\Delta/2\asymp(\log T/T)^{\gamma_1/(2\gamma_1+1)}$ with probability at least $3/4-o(1)$ at a point with heavier memory. Together they show that free adaptation between memory classes is impossible and that a logarithmic loss suffices in probability; they do not identify a sharp adaptive minimax rate under a common loss. The lower bound is a constrained-risk argument \citep{Lepski1990,BrownLow1996}, in which the Kullback--Leibler budget is of order $\log T$ and the unobserved past is again controlled by coupling with the dynamics without slow modes.

\begin{example}\label{ex:price}
The assumptions of Theorem~\ref{thm:price} hold for Omori-type measures $\nu^{(i)}(\dd r)=\rho_i\Gamma(\gamma_i)^{-1}r^{\gamma_i-1}e^{-r}\dd r$, $i=1,2$, provided $\rho_1+\rho_2<\rho_{\max}$, $\rho_1$ is small enough that $\mu_s-\mu_s\rho_1/(1-\rho_2)\ge\mu_{\min}$ for some $\mu_s\in[\mu_{\min},\mu_{\max}]$, and $c_+$ is large enough: for $0<r\le r_\star\le1$, $\nu^{(2)}((0,r])\le C_2r^{\gamma_2}$ and $\nu^{(2)}((0,r])+\nu^{(1)}((0,r])\le(C_1+C_2)r^{\gamma_1}$.
\end{example}

\parhead{No adaptive honest intervals.} The same construction, with a Kullback--Leibler budget of order one instead of $\log T$, shows that honest intervals cannot adapt to lighter memory.

\begin{proposition}\label{prop:noadaptci}
Under the assumptions of Theorem~\ref{thm:price}, let $\alpha<(1-8^{-1/2})/2$ and let $I_T$ be an interval with $\inf_{\mathcal C_{\gamma_1}}\PP(\rho\in I_T)\ge1-\alpha$. Then there are points $P_{s,T}\in\mathcal C_{\gamma_2}$ with
\[
\E_{P_{s,T}}|I_T|\ \ge\ c\,T^{-\gamma_1/(2\gamma_1+1)}\qquad\text{for $T$ large},
\]
which is polynomially larger than the lower-bound scale $T^{-\gamma_2/(2\gamma_2+1)}$ of Proposition~\ref{prop:cilength} for $\mathcal C_{\gamma_2}$, and larger than the near-minimax honest lengths $T^{-\gamma_2/(2\gamma_2+1)}\ell_T^{1/2}$ achieved over $\mathcal C_{\gamma_2}$ by Corollary~\ref{cor:ci} for any sufficiently slowly growing $\ell_T$.
\end{proposition}

An interval that is honest over the heavier-memory class must therefore be as long as the heavier-memory rate even when the data come from the lighter-memory class. This is the point-process analogue of the impossibility of adaptive honest intervals in nonparametric regression \citep{CaiLow2004}.

\section{Proofs}\label{sec:proofs}

\subsection{Two-point construction}

Fix $r_0\in(0,r_\star)$ and let $(\mu_0,\nu_0)$ be as in Theorem~\ref{thm:main}. Define
\begin{equation}\label{eq:alt}
m=\nu_0((0,r_0]),\qquad \nu_1=\nu_0|_{(r_0,\infty)},\qquad \mu_1=\mu_0+\lb\,m,\qquad \lb=\frac{\mu_0}{1-\rho_0}.
\end{equation}

\begin{lemma}\label{lem:twopoint}
\begin{enumerate}
\item[(a)] $\rho_1=\rho_0-m$ and the mean intensities coincide: $\mu_1/(1-\rho_1)=\lb$.
\item[(b)] For $m$ small enough, $(\mu_1,\nu_1)\in\CH$.
\item[(c)] The kernel $\phi_1$ of $\nu_1$ has only relaxation rates $r>r_0$.
\end{enumerate}
\end{lemma}

\begin{proof}
(a) $\mu_1/(1-\rho_1)=(\mu_0+\lb m)/(1-\rho_0+m)=\lb(1-\rho_0+m)/(1-\rho_0+m)$. (b) Removing mass does not increase $\nu((0,r])$, $\nu((0,\infty))$ or $\int r\,\dd\nu$, and $\mu_1\le\mu_{\max}$ once $\lb m\le\mu_{\max}-\mu_0$. (c) is immediate. \qedhere
\end{proof}

Write $\phi_{\rm tail}=\phi_0-\phi_1$, $P_j$ for the stationary law with parameters $(\mu_j,\nu_j)$, and $\lambda_j$ for the corresponding intensity functionals~\eqref{eq:hawkes} evaluated on a common path.

\subsection{The Kullback--Leibler rate}

Let $\varphi(x,y)=x\log(x/y)-x+y\ge0$ for $x,y>0$. Under $P_0$, the process $t\mapsto(\lambda_0(t),\lambda_1(t))$ is stationary, and we define the \emph{Kullback--Leibler rate}
\[
h=\E_0\,\varphi\big(\lambda_0(0),\lambda_1(0)\big).
\]

\begin{proposition}\label{prop:rate}
$\displaystyle h\le\frac{m^2\,r_0}{(1-\rho_0)^3}$.
\end{proposition}

\begin{proof}
Since $\varphi(x,y)\le(x-y)^2/y$ and $\lambda_1\ge\mu_1\ge\mu_0=\lb(1-\rho_0)$,
\[
h\le\frac{\E_0[(\lambda_1-\lambda_0)^2]}{\lb(1-\rho_0)} .
\]
Now $\lambda_1-\lambda_0=\lb m-Y$ with $Y=\int\phi_{\rm tail}(t-s)N(\dd s)$ and $\E_0Y=\lb m$, so by~\eqref{eq:varformula} and~\eqref{eq:bartlett}
\[
\E_0[(\lambda_1-\lambda_0)^2]=\frac{1}{2\pi}\int_\R|\widehat\phi_{\rm tail}(\omega)|^2\,\frac{\lb}{|1-\widehat\phi_0(\omega)|^2}\,\dd\omega
\le\frac{\lb}{(1-\rho_0)^2}\int_0^\infty\phi_{\rm tail}(t)^2\,\dd t,
\]
using $|1-\widehat\phi_0(\omega)|\ge1-|\widehat\phi_0(\omega)|\ge1-\rho_0$ and Plancherel's identity. Finally, $\phi_{\rm tail}$ is non-increasing with $\phi_{\rm tail}(0)=\int_{(0,r_0]}r\,\nu_0(\dd r)\le r_0m$ and $\int\phi_{\rm tail}=m$, so $\int\phi_{\rm tail}^2\le r_0m^2$. \qedhere
\end{proof}

\subsection{Controlling the unobserved past}

\begin{lemma}\label{lem:coupling}
Let $\delta=r_0(1-\rho_1)/2$. For every $L>0$,
\begin{equation}\label{eq:coupling}
\TV\big(P_0^T,P_1^T\big)\le\sqrt{\tfrac12(T+L)\,h}\;+\;\frac{4\rho_1\lb}{(1-\rho_1)\,\delta}\,e^{-\delta L}.
\end{equation}
\end{lemma}

\begin{proof}
Let $P_1'$ be the law of the point process that coincides with the $P_0$-process on $(-\infty,-L]$ and, on $(-L,\infty)$, has intensity $\mu_1+\int\phi_1(t-s)N(\dd s)$, the integral running over the whole past, including points before $-L$.

\emph{Step 1: $P_0$ versus $P_1'$.} Let $\mathcal H_{-L}$ be the history up to time $-L$. The two laws agree on $\mathcal H_{-L}$, so by the chain rule
\[
\KL\big(P_0|_{(-\infty,T]}\,\big\|\,P_1'|_{(-\infty,T]}\big)=\E_0\Big[\KL\big(P_0(\cdot\mid\mathcal H_{-L})\,\big\|\,P_1'(\cdot\mid\mathcal H_{-L})\big)\Big],
\]
where the conditional laws are those of the points in $(-L,T]$. Given $\mathcal H_{-L}$, both conditional laws are described on $(-L,T]$ by intensities with respect to the same full-history filtration, $\lambda_0$ and $\lambda_1\ge\mu_1>0$. By the likelihood formula for point processes \citep{Jacod1975}, the conditional divergence equals $\E_0[\int_{-L}^{T}\varphi(\lambda_0(t),\lambda_1(t))\,\dd t\mid\mathcal H_{-L}]$. Taking expectations and using stationarity of $P_0$ gives $(T+L)\,h$. Restricting to $[0,T]$ does not increase the divergence, and Pinsker's inequality gives $\TV(P_0^T,P_1'^T)\le\sqrt{(T+L)h/2}$.

\emph{Step 2: $P_1'$ versus $P_1$.} Both laws follow the $P_1$ dynamics on $(-L,\infty)$. By the cluster representation, the restriction to $[0,T]$ of either process is the superposition of (i) the clusters of immigrants arriving in $(-L,T]$ and (ii) the descendants, through cascades of $\phi_1$, of points located before $-L$. Part (i) has the same law under $P_1'$ and $P_1$ and is coupled identically. Hence the total variation is at most the probability that part (ii) places a point in $[0,T]$ in either process.

By Lemma~\ref{lem:twopoint}(c) and since $r\mapsto r/(r-\delta)$ decreases on $(r_0,\infty)$,
\[
q:=\int_0^\infty\phi_1(t)e^{\delta t}\,\dd t=\int\frac{r}{r-\delta}\,\nu_1(\dd r)\le\frac{\rho_1\,r_0}{r_0-\delta}=\frac{2\rho_1}{1+\rho_1}<1 .
\]
Thus $\int\Psi_1(u)e^{\delta u}\dd u\le q/(1-q)=2\rho_1/(1-\rho_1)$ and, by Markov's inequality, $\Psi_1([v,\infty))\le\frac{2\rho_1}{1-\rho_1}e^{-\delta v}$. Points before $-L$ have mean density $\lb$ under both laws. Attaching to each of them its full $\phi_1$-cascade over-counts part (ii): some chains pass through times before $-L$, where the process is not driven by $\phi_1$ in $P_1'$, and in $P_1$ a point before $-L$ may be counted both on its own and as a descendant of an earlier point. Over-counting only strengthens the bound, so the expected number of points of part (ii) in $[0,\infty)$ is at most
\[
\lb\int_L^\infty\Psi_1([v,\infty))\,\dd v\le\frac{2\rho_1\lb}{(1-\rho_1)\delta}\,e^{-\delta L}
\]
for each process. Combining both steps with the triangle inequality proves~\eqref{eq:coupling}. \qedhere
\end{proof}

\begin{remark}
The key property is Lemma~\ref{lem:twopoint}(c): the alternative has no relaxation rates at or below $r_0$ and therefore forgets its past exponentially fast. A two-sided tail condition in the class would force one to keep part of the slow mass in the alternative and would destroy this property. This is why $\Cg$ is defined through the upper bound~\eqref{eq:tailcond} only.
\end{remark}

\subsection{Proof of Theorem~\ref{thm:main}}

Let $r_0=r_T$ and $L=L(T)=2\log T/(r_T(1-\rho_1))$. By~\eqref{eq:nulltail}, $c_0H(r_T)\le m\le C_0H(r_T)$, so Proposition~\ref{prop:rate} and~\eqref{eq:rT} give $Th\le T\,C_0^2H(r_T)^2r_T/(1-\rho_0)^3=1/4$. By~\eqref{eq:rT}, $Tr_T=\kappa_0/H(r_T)^2$, hence
\[
\frac{L}{T}=\frac{2\,H(r_T)^2\log T}{\kappa_0(1-\rho_1)}\to0
\]
by assumption, and $(T+L)h\le\tfrac14(1+o(1))$. Moreover $\delta L=\log T$, so the second term in~\eqref{eq:coupling} is $O\big(1/(r_TT)\big)=O\big(H(r_T)^2\big)=o(1)$. Hence $\TV(P_0^T,P_1^T)\le 8^{-1/2}+o(1)$. Since $\mu_0<\mu_{\max}$ and $m\to0$, Lemma~\ref{lem:twopoint}(b) shows that both hypotheses lie in $\CH$ for $T$ large, and $|\rho_0-\rho_1|=m$. Le Cam's two-point lemma \citep[Section~2.4]{Tsybakov2009} yields, for every estimator,
\[
\max_{j=0,1}P_j\big(|\hat\rho-\rho_j|\ge m/2\big)\ge\frac{1-\TV(P_0^T,P_1^T)}{2}\ge\frac14\quad\text{for $T$ large},
\]
so $\mathcal R_T(\CH)\ge m/8\ge\tfrac18c_0H(r_T)$ for $T$ large, which is~\eqref{eq:mainbound}. Corollary~\ref{cor:power} follows by solving~\eqref{eq:rT}; the condition $H(r_T)^2\log T\to0$ holds because $H(r_T)^2\asymp T^{-2\gamma/(2\gamma+1)}$. \qed

\subsection{Proof of Theorem~\ref{thm:upper}}\label{sec:proofupper}

\begin{lemma}[Spectral bias]\label{lem:bias}
There is $B<\infty$, depending only on $\gamma$ and the constants of $\Cg$, such that $|f(\omega)-f(0)|\le B|\omega|^\gamma$ for all $\omega\in\R$ and all $(\mu,\nu)\in\Cg$. Consequently, for every $W>0$,
\[
\Big|\frac{\mathrm{Var}\,N((0,W])}{W}-f(0)\Big|\le B\,\kappa_\gamma\,W^{-\gamma},\qquad \kappa_\gamma=\frac1{2\pi}\int_\R|u|^{\gamma-2}\,4\sin^2(u/2)\,\dd u<\infty .
\]
\end{lemma}

\begin{proof}
Since $\widehat\phi(0)-\widehat\phi(\omega)=\int\nu(\dd r)\,i\omega/(r+i\omega)$, we have $|\widehat\phi(0)-\widehat\phi(\omega)|\le\int\nu(\dd r)\min(1,|\omega|/r)$. For $|\omega|<r_\star$, split at $r=|\omega|$. The part $r\le|\omega|$ is at most $c_+|\omega|^\gamma$. For the part $r>|\omega|$, integration by parts and~\eqref{eq:tailcond} give
\[
|\omega|\int_{(|\omega|,\infty)}\frac{\nu(\dd r)}{r}\le|\omega|\Big(c_+r_\star^{\gamma-1}+c_+\int_{|\omega|}^{r_\star}r^{\gamma-2}\dd r+\frac{\rho_{\max}}{r_\star}\Big)\le B_1|\omega|^\gamma .
\]
For $|\omega|\ge r_\star$ the difference is at most $2\rho_{\max}\le(2\rho_{\max}r_\star^{-\gamma})|\omega|^\gamma$. Finally, $|1-\widehat\phi|\ge1-\rho_{\max}$ and $|1-\widehat\phi|\le2$ give
\[
|f(\omega)-f(0)|=\lb\,\frac{\big||1-\widehat\phi(0)|^2-|1-\widehat\phi(\omega)|^2\big|}{|1-\widehat\phi(\omega)|^2|1-\widehat\phi(0)|^2}\le\frac{4\lb}{(1-\rho_{\max})^4}\,|\widehat\phi(0)-\widehat\phi(\omega)| .
\]
For the second statement, recall \citep{DaleyVereJones2003} that
\[
\mathrm{Var}\,N((0,W])=\frac1{2\pi}\int_\R f(\omega)\,\frac{4\sin^2(\omega W/2)}{\omega^2}\,\dd\omega,\qquad \frac1{2\pi}\int_\R\frac{4\sin^2(\omega W/2)}{\omega^2}\,\dd\omega=W .
\]
Hence the left-hand side equals $\frac1{2\pi W}\int(f(\omega)-f(0))\,4\sin^2(\omega W/2)\,\omega^{-2}\dd\omega$, and the substitution $u=\omega W$ gives the bound; $\kappa_\gamma<\infty$ because $\gamma\in(0,1)$. \qedhere
\end{proof}

\begin{lemma}[Stochastic error]\label{lem:var}
Let $S$ be the total size of a cluster (immigrant and all descendants) and $s_k=\mu\,\E[S^k]$. With $Y_j=N_j-W\lb$ and $V_0=T'^{-1}\sum_jY_j^2$,
\[
\E V_0=\frac{\mathrm{Var}\,N((0,W])}{W},\qquad \mathrm{Var}(V_0)\le\frac{2Ws_2^2+s_4}{T'},\qquad \mathrm{Var}(\hat\lambda)\le\frac{s_2}{T'},
\]
\[
\E(\hat\lambda-\lb)^4\le\frac{3s_2^2}{T'^2}+\frac{s_4}{T'^3},\qquad \mathrm{Var}(\hat V)\le\frac{2(2Ws_2^2+s_4)}{T'}+2W^2\Big(\frac{3s_2^2}{T'^2}+\frac{s_4}{T'^3}\Big)\le\frac{C_VW}{T'},
\]
with $C_V=10s_2^2+4s_4$, which is bounded uniformly over the class. Moreover $\sup_{\rho\le\rho_{\max}}\E[S^k]<\infty$ for every $k$.
\end{lemma}

\begin{proof}
\emph{Cluster sizes.} In the cluster representation, $S$ is the total progeny of a Galton--Watson process with Poisson($\rho$) offspring, which has the Borel law $\PP(S=n)=e^{-\rho n}(\rho n)^{n-1}/n!$. Since $\rho e^{1-\rho}<1$ for $\rho<1$, Stirling's formula shows that $\PP(S=n)$ decays geometrically, so $S$ has moments of all orders. A standard coupling of Galton--Watson trees shows that $S$ is stochastically increasing in $\rho$, hence $\E[S^k]\le\E_{\rho_{\max}}[S^k]<\infty$.

\emph{Cumulants.} Let $N_s$ denote the (translated) cluster whose immigrant is at $s$. The Laplace functional of the Poisson cluster process is $\log\E\exp(-\int g\,\dd N)=\mu\int\big(\E\exp(-\int g\,\dd N_s)-1\big)\dd s$ for measurable $g\ge0$ \citep[Section~6.3]{DaleyVereJones2003}. Taking $g=\sum_it_i\mathbf 1_{A_i}$ and differentiating at $t=0$ gives, for bounded sets $A_1,\dots,A_k$,
\begin{equation}\label{eq:clustercum}
\kappa\big(N(A_1),\dots,N(A_k)\big)=\mu\int_\R\E\Big[\prod_{i=1}^kN_s(A_i)\Big]\dd s\ \ge0 .
\end{equation}

\emph{A moment bound.} Pathwise, $\int_\R N_s(A)\,\dd s=\sum_{x}\int\mathbf 1\{s+x\in A\}\dd s=S|A|$, the sum running over the relative positions $x$ of the $S$ cluster points; and $N_s(A)\le S$, so $N_s(A)^k\le S^{k-1}N_s(A)$. Hence
\[
\mu\int_\R\E[N_s(A)^k]\,\dd s\le\mu\,\E\Big[S^{k-1}\int_\R N_s(A)\,\dd s\Big]=|A|\,s_k .
\]
With $k=2$ and~\eqref{eq:clustercum} this gives $\mathrm{Var}\,N(A)\le|A|s_2$, hence the bound on $\hat\lambda$, and $\E V_0=\mathrm{Var}\,N((0,W])/W$ by stationarity.

\emph{Variance of $V_0$.} For centred variables, $\mathrm{Cov}(Y_j^2,Y_k^2)=2\,\mathrm{Cov}(Y_j,Y_k)^2+\kappa(Y_j,Y_j,Y_k,Y_k)$, and cumulants of order at least two are unchanged by centring. By~\eqref{eq:clustercum} all covariances are non-negative, and $\mathrm{Cov}(Y_j,Y_k)\le\max_i\mathrm{Var}\,Y_i\le Ws_2$. Therefore
\[
\sum_{j,k}\mathrm{Cov}(Y_j,Y_k)^2\le Ws_2\sum_{j,k}\mathrm{Cov}(Y_j,Y_k)=Ws_2\,\mathrm{Var}\,N((0,T'])\le WT's_2^2 .
\]
For the fourth-order term, by~\eqref{eq:clustercum},
\[
\sum_{j,k}\kappa(Y_j,Y_j,Y_k,Y_k)=\mu\int\E\Big[\Big(\sum_jN_s(B_j)^2\Big)^2\Big]\dd s\le\mu\int\E\big[N_s((0,T'])^4\big]\dd s\le T's_4,
\]
using $\sum_jN_s(B_j)^2\le N_s((0,T'])^2$. Hence $\mathrm{Var}(\sum_jY_j^2)\le2WT's_2^2+T's_4$, and dividing by $T'^2$ gives the bound on $\mathrm{Var}(V_0)$.

\emph{Variance of $\hat V$.} For the count $N((0,T'])$, $\E(N-\E N)^4=3\,\mathrm{Var}(N)^2+\kappa_4(N)\le3T'^2s_2^2+T's_4$ by~\eqref{eq:clustercum} and the moment bound, which gives the bound on $\E(\hat\lambda-\lb)^4$ after division by $T'^4$. Since $\hat V=V_0-W(\hat\lambda-\lb)^2$, $\mathrm{Var}(\hat V)\le2\,\mathrm{Var}(V_0)+2W^2\E(\hat\lambda-\lb)^4$, and $W\le T'$ gives the last inequality. \qedhere
\end{proof}

\begin{proof}[Proof of Theorem~\ref{thm:upper}]
Since $\sum_jY_j=T'(\hat\lambda-\lb)$, we have $\hat V=V_0-W(\hat\lambda-\lb)^2$, and $\E[W(\hat\lambda-\lb)^2]\le Ws_2/T'$. By Lemmas~\ref{lem:bias} and~\ref{lem:var},
\[
\E|\hat V-f(0)|\le B\kappa_\gamma W^{-\gamma}+\sqrt{\frac{2Ws_2^2+s_4}{T'}}+\frac{Ws_2}{T'},\qquad \E|\hat\lambda-\lb|\le\sqrt{\frac{s_2}{T'}} .
\]
The true values satisfy $\lb\in[\mu_{\min},\lambda_{\max}]$, $f(0)=\lb/(1-\rho)^2\in[\mu_{\min},v_{\max}]$ and $\rho\in[0,\rho_{\max}]$, so neither the clipping nor the final projection increases the errors. On the rectangle $[\mu_{\min},\lambda_{\max}]\times[\mu_{\min},v_{\max}]$ the map $(\lambda,v)\mapsto1-\sqrt{\lambda/v}$ is Lipschitz with a constant depending only on the class, and $\rho=1-\sqrt{\lb/f(0)}$. Hence $\E|\hat\rho-\rho|$ is bounded by a constant times the sum of the two displayed errors. With $W=W_T$ and $T'\ge T/2$, each term is $O(T^{-\gamma/(2\gamma+1)})$, uniformly over $\Cg$. \qedhere
\end{proof}

\subsection{Proofs of Theorem~\ref{thm:clt} and its consequences}\label{sec:proofclt}

\parhead{Poisson embedding.} Let $\mathcal K$ be the space of finite point configurations on $[0,\infty)$ containing $0$, and $Q$ the law of a cluster seen from its immigrant. The stationary Hawkes process is $N=\sum_{(s,\kappa)\in\eta}\delta_{s+\kappa}$, where $\eta$ is a Poisson process on $\mathbb X=\R\times\mathcal K$ with intensity $\lambda(\dd s,\dd\kappa)=\mu\,\dd s\,Q(\dd\kappa)$ \citep{HawkesOakes1974}. Hence $F=\sum_{j=1}^MY_j^2$ is a square-integrable Poisson functional in the domain of $D$, and we can use the second order Poincar\'e inequality of \citet[Theorem~1.1]{LastPeccatiSchulte2016}: for $\tilde F=(F-\E F)/\sigma_W$,
\[
d_{\mathrm W}(\tilde F,\mathcal N(0,1))\le\gamma_1+\gamma_2+\gamma_3,
\]
where, with $D$ the difference (add-one-cost) operator,
\begin{align*}
\gamma_1&=\frac{2}{\sigma_W^2}\Big[\int\big(\E[(D_{x_1}F)^2(D_{x_2}F)^2]\big)^{1/2}\big(\E[(D^2_{x_1,x_3}F)^2(D^2_{x_2,x_3}F)^2]\big)^{1/2}\lambda^3(\dd x)\Big]^{1/2},\\
\gamma_2&=\frac{1}{\sigma_W^2}\Big[\int\E[(D^2_{x_1,x_3}F)^2(D^2_{x_2,x_3}F)^2]\,\lambda^3(\dd x)\Big]^{1/2},\qquad
\gamma_3=\frac{1}{\sigma_W^3}\int\E|D_xF|^3\,\lambda(\dd x).
\end{align*}

\parhead{Difference operators.} For $x=(s,\kappa)$ write $n_j(x)$ for the number of points of $s+\kappa$ in $B_j$ and $S_x=|\kappa|$. Adding a cluster adds $n_j(x)$ to each $N_j$, so
\[
D_xF=\sum_j\big(2Y_jn_j(x)+n_j(x)^2\big),\qquad D^2_{x,y}F=2\sum_jn_j(x)\,n_j(y),
\]
and $D^2F$ is deterministic. We use three facts: $\sum_jn_j(x)\le S_x$; $\int_\R n_j(s,\kappa)\,\dd s=|\kappa|\,W$ for every block; and, from Lemma~\ref{lem:var}, $\E Y_j^4\le3W^2s_2^2+Ws_4\le CW^2$ and $\E[S^k]\le C_k$ uniformly over $\Cg$. Constants $C$ below depend only on the class and may change from line to line.

\parhead{Variance.} By~\eqref{eq:clustercum} all terms in $\sigma_W^2=\sum_{j,k}\{2\,\mathrm{Cov}(Y_j,Y_k)^2+\kappa(Y_j,Y_j,Y_k,Y_k)\}$ are non-negative. Since $\mathrm{Var}\,N(A)\ge\mu\int\E N_s(A)\,\dd s=\lb|A|$, we get the lower bound $\sigma_W^2\ge2M(\lb W)^2\ge2\mu_{\min}^2\,T'W$. For the asymptotics, $\mathrm{Var}\,Y_j=Wf(0)(1+O(W^{-\gamma}))$ by Lemma~\ref{lem:bias}. Moreover, the covariance measure of $N$, including its diagonal part, is non-negative and has total mass $f(0)$ per unit time, so $\sum_k\mathrm{Cov}(Y_j,Y_k)\le Wf(0)$, so $\sum_{k\ne j}\mathrm{Cov}(Y_j,Y_k)\le Wf(0)-\mathrm{Var}\,Y_j=O(W^{1-\gamma})$ and $\sum_{j\neq k}\mathrm{Cov}(Y_j,Y_k)^2\le M\cdot Wf(0)\cdot O(W^{1-\gamma})$. The fourth-cumulant part is at most $T's_4=O(MW)$ by Lemma~\ref{lem:var}. Hence $\sigma_W^2=2MW^2f(0)^2(1+O(W^{-\gamma}+W^{-1}))$, which is the second claim of Theorem~\ref{thm:clt}.

\parhead{Bound on $\gamma_3$.} Since $(a+b)^3\le4(a^3+b^3)$, and by Jensen's inequality with weights $n_j/S_x$, $(\sum_j|Y_j|n_j)^3\le S_x^2\sum_jn_j|Y_j|^3$, we have
\[
\E|D_xF|^3\le C\Big(S_x^2\sum_jn_j(x)\,\E|Y_j|^3+S_x^5\sum_jn_j(x)\Big)\le C\,W^{3/2}S_x^5\sum_jn_j(x),
\]
because the added cluster is independent of $\eta$ and $\E|Y_j|^3\le CW^{3/2}$. Integrating, $\int\sum_jn_j(s,\kappa)\,\dd s=|\kappa|T'$, so $\int\E|D_xF|^3\lambda(\dd x)\le C\mu\E[S^6]\,T'W^{3/2}$ and $\gamma_3\le CT'W^{3/2}/(T'W)^{3/2}=C\,T'^{-1/2}$.

\parhead{Bound on $\gamma_2$.} For fixed $x_3$, by Cauchy--Schwarz and $\sum_jn_j(x_1)\le S_{x_1}$,
\[
\int(D^2_{x_1,x_3}F)^2\lambda(\dd x_1)\le4\int S_{x_1}\sum_jn_j(x_1)\,n_j(x_3)^2\,\lambda(\dd x_1)=4\mu\E[S^2]\,W\sum_jn_j(x_3)^2\le4\mu\E[S^2]\,W\,S_{x_3}^2 .
\]
The same holds for $x_2$, and $\int S_{x_3}^4\mathbf 1\{n(x_3)\neq0\}\lambda(\dd x_3)\le\int S_{x_3}^4\sum_jn_j(x_3)\lambda(\dd x_3)=\mu\E[S^5]T'$. Hence the integral in $\gamma_2$ is at most $CW^2T'$, and $\gamma_2\le C\,W\sqrt{T'}/(T'W)=C\,T'^{-1/2}$.

\parhead{Bound on $\gamma_1$.} By Cauchy--Schwarz, $\E[(D_{x_1}F)^2(D_{x_2}F)^2]^{1/2}\le(\E(D_{x_1}F)^4)^{1/4}(\E(D_{x_2}F)^4)^{1/4}$, and as for $\gamma_3$, $(\E(D_xF)^4)^{1/4}\le C\,W^{1/2}S_x^2$. Since $D^2F$ is deterministic, the integrand is at most $CW\,S_{x_1}^2S_{x_2}^2\,|D^2_{x_1,x_3}F|\,|D^2_{x_2,x_3}F|$. Now $\int S_{x_1}^2|D^2_{x_1,x_3}F|\lambda(\dd x_1)\le2\mu\E[S^3]\,W\,S_{x_3}$, and similarly for $x_2$, so the integral is at most $CW^3\int S_{x_3}^2\mathbf 1\{n(x_3)\ne0\}\lambda(\dd x_3)\le CW^3T'$. Hence $\gamma_1\le C\,(W^3T')^{1/2}/(T'W)=C\,(W/T')^{1/2}$. Together with $T'\ge T/2$ this proves Theorem~\ref{thm:clt}.

\parhead{Proof of Corollary~\ref{cor:an}.} Write $V_0=F/T'$ and $f_0=f(0)$. By Theorem~\ref{thm:clt} and the inequality $d_{\mathrm K}\le2\sqrt{d_{\mathrm W}}$ for a standard normal target \citep{NourdinPeccati2012}, $Z_T=\sqrt{T'/(2W)}\,(V_0-\E V_0)/f_0$ satisfies $\sup_{\Cg}d_{\mathrm K}(Z_T,\mathcal N(0,1))\to0$. Next, $\hat V-f_0=(V_0-\E V_0)+(\E V_0-f_0)-W(\hat\lambda-\lb)^2$, where $|\E V_0-f_0|\le B\kappa_\gamma W^{-\gamma}$ (Lemma~\ref{lem:bias}) and $\E[W(\hat\lambda-\lb)^2]\le Ws_2/T'$. After multiplication by $\sqrt{T'/W}$ these are $O(\sqrt{T/W^{2\gamma+1}})=o(1)$ and $O_P(\sqrt{W/T})=o_P(1)$ uniformly, and $\sqrt{T'/W}\,|\hat\lambda-\lb|=O_P(W^{-1/2})$. On the event $\{|\hat V-f_0|\le f_0/2,\ |\hat\lambda-\lb|\le\lb/2\}$, whose probability tends to one uniformly by Chebyshev's inequality, a Taylor expansion of $g(\lambda,v)=1-\sqrt{\lambda/v}$ with $\partial_vg=(1-\rho)/(2f_0)$ gives
\[
\sqrt{\tfrac{2T'}{W}}\,\frac{\check\rho-\rho}{1-\rho}=Z_T+R_T,\qquad \sup_{\Cg}\PP(|R_T|>\varepsilon)\to0\ \text{for every }\varepsilon>0,
\]
and $(1-\check\rho)/(1-\rho)\to1$ in probability uniformly, since $1-\rho\ge1-\rho_{\max}$. The claim follows from $d_{\mathrm K}(Z+R,\mathcal N(0,1))\le d_{\mathrm K}(Z,\mathcal N(0,1))+\PP(|R|>\varepsilon)+\varepsilon$. \qed

\parhead{Proof of Corollary~\ref{cor:ci}.} The choice of $W_T$ satisfies the conditions of Corollary~\ref{cor:an}, since $TW_T^{-(2\gamma+1)}=\ell_T^{-(2\gamma+1)}\to0$ and $W_T/T\to0$. Since $f(0)\ge\lb\ge\mu_{\min}$, Chebyshev's inequality, Lemma~\ref{lem:var} and Lemma~\ref{lem:bias} give
\[
\PP(\mathcal R^c)\le\frac{4\{\mathrm{Var}(\hat V)+(\E\hat V-f(0))^2\}}{f(0)^2}+\frac{4\,\mathrm{Var}(\hat\lambda)}{\lb^2}\le C\Big(\frac{W_T}{T}+W_T^{-2\gamma}\Big),
\]
uniformly over $\Cg$. On $\mathcal R^c$ the interval contains $\rho$, and on $\mathcal R$ the clipping of $\hat s$ is inactive with probability tending to one uniformly, since $1-\rho\in[1-\rho_{\max},1]$ and $\sqrt{\hat\lambda/\hat V}\to1-\rho$ uniformly in probability. Coverage therefore follows from Corollary~\ref{cor:an}. For the length, $|I_T|\le2z_{1-\alpha/2}\sqrt{W_T/(2T')}$ on $\mathcal R$ and $|I_T|\le1$ on $\mathcal R^c$, so
\[
\E|I_T|\le2z_{1-\alpha/2}\sqrt{W_T/T}+C\big(W_T/T+W_T^{-2\gamma}\big)\le C\,T^{-\gamma/(2\gamma+1)}\ell_T^{1/2},
\]
because $W_T/T\le\sqrt{W_T/T}$ and $W_T^{-2\gamma}=T^{-2\gamma/(2\gamma+1)}\ell_T^{-2\gamma}\le\sqrt{W_T/T}$. \qed

\parhead{Proof of Proposition~\ref{prop:cilength}.} Let $P_0,P_1$ be the two points of Section~\ref{sec:proofs} with $|\rho_0-\rho_1|=m\ge c_0r_T^\gamma$ and $\tau=\TV(P_0^T,P_1^T)\le8^{-1/2}+o(1)$. On the event $A=\{|I_T|<m\}$ the interval cannot contain both $\rho_0$ and $\rho_1$. Hence $P_1(\rho_1\in I_T,A)\le P_0(\rho_1\in I_T,A)+\tau\le P_0(\rho_0\notin I_T)+\tau\le\alpha+\tau$, and $P_1(A^c)\ge P_1(\rho_1\in I_T)-P_1(\rho_1\in I_T,A)\ge1-2\alpha-\tau$. Therefore $\E_1|I_T|\ge m(1-2\alpha-\tau)$, which is bounded below by a positive multiple of $T^{-\gamma/(2\gamma+1)}$ when $\alpha<(1-8^{-1/2})/2$. \qed

\subsection{Proofs of Theorems~\ref{thm:adapt} and~\ref{thm:price}}\label{sec:proofadapt}

\parhead{Proof of Theorem~\ref{thm:adapt}.} Constants below are uniform over $\gamma\in[\gamma_-,\gamma_+]$ and $\Cg$; in particular the constant $B\kappa_\gamma$ of Lemma~\ref{lem:bias} is bounded on $[\gamma_-,\gamma_+]$. Let $b_k=B\kappa_\gamma W_k^{-\gamma}$. By Lemma~\ref{lem:bias} and the identity $\hat V=V_0-W(\hat\lambda-\lb)^2$, $|\E\hat V_k-f(0)|\le b_k+W_ks_2/T'\le b_k+\bar\sigma_k$ for $T$ large, since $W_k/T\le\sqrt{W_k/T}$.

\emph{Good event.} Let $G$ be the event that $|\hat V_k-\E\hat V_k|\le x_T\bar\sigma_k$ and $|\hat\lambda_k-\lb|\le x_T\sqrt{s_2/T'_k}$ for all $k$. By Chebyshev's inequality, Lemma~\ref{lem:var} and a union bound over the $K_T-k_0+1=O(\log T)$ grid points, $\PP(G^c)\le C\log T/x_T^2=O((\log T)^{-\varepsilon})$, uniformly.

\emph{Oracle index.} Let $k^*=\min\{k:\ b_k\le x_T\bar\sigma_k\}$. Since $b_k$ decreases and $\bar\sigma_k$ increases in $k$, $b_l\le x_T\bar\sigma_l$ for all $l\ge k^*$. On $G$, for $l>k\ge k^*$,
\[
|\hat V_k-\hat V_l|\le|\hat V_k-f(0)|+|\hat V_l-f(0)|\le(b_k+2x_T\bar\sigma_k)+(b_l+2x_T\bar\sigma_l)\le6x_T\bar\sigma_l,
\]
so $k^*$ satisfies the condition in~\eqref{eq:lepski} and $\hat k\le k^*$. If $\hat k<k^*$, the rule applied with $l=k^*$ gives $|\hat V_{\hat k}-\hat V_{k^*}|\le6x_T\bar\sigma_{k^*}$, hence $|\hat V_{\hat k}-f(0)|\le6x_T\bar\sigma_{k^*}+3x_T\bar\sigma_{k^*}=9x_T\bar\sigma_{k^*}$; the same bound holds trivially if $\hat k=k^*$.

\emph{Size of the oracle error.} Let $W^*$ solve $B\kappa_\gamma W^{-\gamma}=x_T\sqrt{C_VW/T}$, that is $W^*\asymp(T/x_T^2)^{1/(2\gamma+1)}$, which lies inside the grid for $T$ large. By minimality of $k^*$, $W_{k^*}\le2W^*$, so $x_T\bar\sigma_{k^*}\le\sqrt2\,x_T\sqrt{C_VW^*/T}\asymp(x_T^2/T)^{\gamma/(2\gamma+1)}$. Finally, on $G$ the map $(\lambda,v)\mapsto1-\sqrt{\lambda/v}$ is Lipschitz on the clipping rectangle, and $|\hat\lambda_{\hat k}-\lb|\le x_T\sqrt{2s_2/T}$ is of smaller order. \qed

\parhead{Proof of Theorem~\ref{thm:price}.} Fix $r_0$ small and set $\nu_s=\nu^{(2)}|_{(r_0,\infty)}$, $\nu_r=\nu_s+\nu^{(1)}|_{(0,r_0]}$, $m=\nu^{(1)}((0,r_0])$, $\mu_r=\mu_s-\lb m$, where $\lb=\mu_s/(1-\rho_s)$. Then $\rho_r=\rho_s+m$, both laws have mean intensity $\lb$, $P_s\in\mathcal C_{\gamma_2}\subset\mathcal C_{\gamma_1}$ and $P_r\in\mathcal C_{\gamma_1}$ by assumption, and $\nu_s$ has only relaxation rates $r>r_0$. Let $\Delta=m\ge c_0r_0^{\gamma_1}$ and $B=\{|\hat\rho-\rho_s|\ge\Delta/2\}$. On $B^c$, $|\hat\rho-\rho_r|>\Delta/2$.

\emph{Small probability under $P_s$.} By assumption and Markov's inequality, $P_s(B)\le2C\,T^{-\gamma_2/(2\gamma_2+1)}/\Delta$.

\emph{Coupling.} Let $P_s'$ coincide with $P_r$ on $(-\infty,-L]$ and follow the $P_s$ dynamics on $(-L,\infty)$. Step~2 of the proof of Lemma~\ref{lem:coupling}, applied to the fast dynamics of $P_s$, gives $\TV(P_s'^T,P_s^T)\le\frac{4\rho_s\lb}{(1-\rho_s)\delta}e^{-\delta L}$ with $\delta=r_0(1-\rho_s)/2$; with $\delta L=2\log T$ this is $O(1/(r_0T^2))$. Hence $q_T:=P_s'(B)\le2C\,T^{-\gamma_2/(2\gamma_2+1)}/\Delta+O(1/(r_0T^2))$.

\emph{Kullback--Leibler budget.} As in Step~1 of Lemma~\ref{lem:coupling}, $\KL(P_r^T\|P_s'^T)\le(T+L)h_{rs}$ with $h_{rs}=\E_r\varphi(\lambda_r(0),\lambda_s(0))$. Here $\lambda_r-\lambda_s=Y-\lb m$ with $Y=\int\phi_{\rm slow}(t-s)N(\dd s)$, $\phi_{\rm slow}$ the kernel of $\nu^{(1)}|_{(0,r_0]}$, and $\E_rY=\lb m$. Arguing as in Proposition~\ref{prop:rate} under $P_r$, with $\lambda_s\ge\mu_s$,
\[
h_{rs}\le\frac{\mathrm{Var}_r(Y)}{\mu_s}\le\frac{\lb\,r_0\,m^2}{(1-\rho_r)^2\mu_s}\le K_1\,r_0^{2\gamma_1+1}
\]
for a constant $K_1$.

\emph{Constrained risk.} For the binary Kullback--Leibler divergence, $\mathrm{kl}(p,q)\ge p\log(1/q)-\log2$. By data processing,
\[
P_r(B)\le\frac{\KL(P_r^T\|P_s'^T)+\log2}{\log(1/q_T)} .
\]
Choose $r_0=r_0(T)$ with $K_1Tr_0^{2\gamma_1+1}=\tfrac b4\log T$, where $b=\frac{\gamma_2}{2\gamma_2+1}-\frac{\gamma_1}{2\gamma_1+1}>0$. Then $\Delta\asymp(\log T/T)^{\gamma_1/(2\gamma_1+1)}$, $L/T\to0$, $\KL\le\frac b4\log T(1+o(1))$ and $\log(1/q_T)=b\log T\,(1+o(1))$. Hence $P_r(B)\le\frac14+o(1)$, and
\[
\E_r|\hat\rho-\rho_r|\ge\frac\Delta2\,P_r(B^c)\ge\Big(\frac38-o(1)\Big)\Delta\asymp\Big(\frac{\log T}{T}\Big)^{\gamma_1/(2\gamma_1+1)} .
\]
Since $P_r\in\mathcal C_{\gamma_1}$, this proves the theorem. \qed

\parhead{Proof of Proposition~\ref{prop:noadaptci}.} Use the construction of the previous proof, but choose $r_0=r_0(T)$ with $K_1Tr_0^{2\gamma_1+1}=\kappa$ for a constant $\kappa$ such that $\sqrt{\kappa/2}\le8^{-1/2}$. Then $\KL(P_r^T\|P_s'^T)\le\kappa(1+o(1))$, so by Pinsker's inequality and the coupling bound, $\tau=\TV(P_r^T,P_s^T)\le8^{-1/2}+o(1)$, and $\Delta=m\ge c_0r_0^{\gamma_1}\asymp T^{-\gamma_1/(2\gamma_1+1)}$. Both $P_s$ and $P_r$ lie in $\mathcal C_{\gamma_1}$, where $I_T$ is honest, and $P_s\in\mathcal C_{\gamma_2}$. The argument of Proposition~\ref{prop:cilength}, with the roles of the two laws chosen so that the expectation is taken under $P_s$, gives $\E_s|I_T|\ge\Delta(1-2\alpha-\tau)$, which is bounded below by a positive multiple of $T^{-\gamma_1/(2\gamma_1+1)}$. \qed

\section{From approximation to an inverse problem}\label{sec:pseudo}

This section explains how Theorem~\ref{thm:main} manifests itself when kernels are fitted as sums of $K$ exponentials, $\phi_\theta(t)=\sum_{k=1}^Ka_kb_ke^{-b_kt}$ with $\rho_\theta=\sum_ka_k$. The statements here are partly formal and are supported by the computations of Section~\ref{sec:numerics}.

\parhead{A quadratic surrogate.} When relative intensity fluctuations are small, expanding $\varphi$ to second order around $\lambda_0$ and using~\eqref{eq:varformula}, with the baseline chosen to match the mean intensity, gives the quadratic functional
\begin{equation}\label{eq:D2}
D_2(\theta)=\frac{1}{4\pi}\int_\R\frac{|\widehat\phi_\theta(\omega)-\widehat\phi_0(\omega)|^2}{|1-\widehat\phi_0(\omega)|^2}\,\dd\omega .
\end{equation}
The mean intensity cancels. $D_2$ is a surrogate for, not equal to, the Kullback--Leibler rate of the misspecified model, and it is a stationary, infinite-horizon criterion that does not involve $T$. We call $\theta^\dagger_{2,K}=\arg\min_\theta D_2(\theta)$ the \emph{quadratic-surrogate pseudo-truth}. It is a weighted $L^2$ fit of the transfer function, with weight equal to the true spectrum.

The branching ratio, however, is a point evaluation, $\rho_\theta=\widehat\phi_\theta(0)$, and point evaluation at zero frequency is not coercively controlled by this weighted $L^2$ norm. A component $a\,b\,e^{-bt}$ changes $\widehat\phi_\theta(0)$ by $a$ but changes $D_2$ only by an amount that vanishes as $b\downarrow0$, since its transfer function concentrates on $|\omega|\lesssim b$. Combined with a finite record, which has effective frequency resolution of order $T^{-1}$, this lack of coercivity becomes the inverse-problem resolution limit quantified by Theorem~\ref{thm:main}.

\parhead{Small $K$: approximation error.} For $K\le4$, $\theta^\dagger_{2,K}$ is well defined and essentially independent of $T$. For the Omori kernel with $\rho_0=0.8$ and $\gamma=0.6$ it gives $\rho_0-\rho^\dagger_{2,K}$ equal to $0.199$, $0.087$, $0.046$ and $0.026$ for $K=1,\dots,4$. The biases of simulated misspecified maximum likelihood fits are $0.186$, $0.082$, $0.044$ and $0.023$ (Table~\ref{tab:K}, 20 replicates), so the surrogate predicts them to within about $10\%$ with no fitted constants; the surrogate is slightly conservative, as expected from a second-order approximation. This regime is governed by the positive rational approximation of a symbol with a branch point at the origin.

\parhead{Larger $K$: a flat slow-pole direction.} As $K$ grows, the fit can afford components with small $b$, and the non-coercive direction appears: slow components exchange mass with the baseline at almost no cost in $D_2$. In the present experiment this happens from about $K=5$ on: the minimiser of~\eqref{eq:D2} is no longer numerically determined, the likelihood of fitted models stops improving, and the estimated baseline absorbs the missing branching mass. The threshold $K\approx5$ is a property of this experiment, not a universal constant. The remaining bias is a resolution limit rather than an approximation error. We regard the transition between the two regimes as demonstrated numerically and explained by the theory; we do not locate it by a theorem.

\section{Numerical illustrations}\label{sec:numerics}

Unless otherwise stated, simulations use the Omori kernel of Example~\ref{ex:omori} with $\gamma=0.6$ and $\mu_0=0.2$, generated by the cluster representation with the rate of each offspring drawn exactly from $\nu$.

\parhead{Oracle conditional likelihood.} To isolate the rate mechanism we evaluate log-likelihood ratios (LLR) with intensities computed from the entire simulated past, including the points before time $0$. This gives the test more information than the observation scheme of Theorem~\ref{thm:main}, where the past is unobserved; Lemma~\ref{lem:coupling} handles that loss analytically. Intensities are evaluated through a discretised Bernstein lift: log-spaced bins of the rate axis, each carrying its exact $\nu$-mass, with the cut $r_0$ placed at a bin edge so that the removed mass is exact. With 160 bins the LLR differs from its 640-bin value by at most $9\times10^{-4}$ (mean $|\mathrm{LLR}|\approx0.6$), and by $1.8\times10^{-4}$ with 320 bins. All results below use 160 bins; the quadrature error at the extremes of the shrinking-difference experiment is reported in the supplement \citep{HerreraMarinSupp}.

\parhead{Burn-in.} Paths are simulated from an empty history at time $-B$, so the null process is not exactly stationary on $[0,T]$. The relative deficit of its mean intensity at time $0$ can be computed exactly through the Bernstein lift. It equals $2.7\times10^{-3}$ for $B=2\times10^5$ and $1.2\times10^{-3}$ for $B=8\times10^5$, that is $13\%$ and $6\%$ of the baseline shift $\lb m$ that separates the hypotheses when $m=0.02$. The deficit lowers the effective baseline of null paths and therefore makes the test easier, so reported test errors are, if anything, slightly optimistic for the detection of the slow alternative. Table~\ref{tab:lecam} uses $B=2\times10^5$ and Table~\ref{tab:tri} uses $B=8\times10^5$.

\subsection{Two-point experiments}

We compare the null ($\rho_0=0.8$) with alternatives having the same mean intensity. For each configuration we simulate independent paths under both laws, generated after the burn-in described above, and report the error of the LLR test with threshold zero,
\[
e_T=\tfrac12\big[P_0(\mathrm{LLR}<0)+P_1(\mathrm{LLR}\ge0)\big],
\]
and the mean LLR under the null, which estimates the conditional Kullback--Leibler divergence.

\parhead{Fixed difference.} Table~\ref{tab:lecam} uses $\Delta\rho=0.02$ and two alternatives: the construction~\eqref{eq:alt}, which removes the $\nu$-mass below $r_0$ ($1/r_0\approx564$), and a uniform rescaling of the kernel, both compensated in the baseline.

\begin{table}
\caption{Two-point experiment with fixed $\Delta\rho=0.02$ and equal mean intensity (30 paths per law).}\label{tab:lecam}
\begin{tabular}{@{}rcccc@{}}
\toprule
 & \multicolumn{2}{c}{Slow mass $\to$ immigration} & \multicolumn{2}{c}{Uniform rescaling}\\
$T$ & $e_T$ & $\E_0[\mathrm{LLR}]$ & $e_T$ & $\E_0[\mathrm{LLR}]$\\
\midrule
$10^{4}$ & $0.50\pm0.06$ & $-0.02\pm0.03$ & $0.33\pm0.06$ & $0.85\pm0.25$\\
$3\cdot10^{4}$ & $0.43\pm0.06$ & $0.09\pm0.05$ & $0.22\pm0.05$ & $1.81\pm0.48$\\
$10^{5}$ & $0.40\pm0.06$ & $0.13\pm0.11$ & $0.02\pm0.02$ & $8.58\pm0.77$\\
\bottomrule
\end{tabular}
\end{table}

What determines the available information is where the difference lies in the relaxation spectrum, not the size of $\Delta\rho$.

\parhead{Shrinking difference.} Table~\ref{tab:tri} follows the construction of the proof: $r_0(T)\propto T^{-1/(2\gamma+1)}$, so that $\Delta\rho_T=m(r_0(T))\propto T^{-\gamma/(2\gamma+1)}$. Theorem~\ref{thm:main} predicts that the test error stays bounded away from zero, and the mean LLR stays bounded, as $T$ grows.

\begin{table}
\caption{Two-point experiment with $\Delta\rho_T\propto T^{-\gamma/(2\gamma+1)}$ (200 paths per law, burn-in $B=8\times10^5$).}\label{tab:tri}
\begin{tabular}{@{}rcccc@{}}
\toprule
$T$ & $1/r_0(T)$ & $\Delta\rho_T$ & $e_T$ & $\E_0[\mathrm{LLR}]$\\
\midrule
$10^{4}$ & $198$ & $0.0374$ & $0.37\pm0.02$ & $0.14\pm0.03$\\
$3\cdot10^{4}$ & $326$ & $0.0278$ & $0.41\pm0.02$ & $0.18\pm0.04$\\
$10^{5}$ & $564$ & $0.0200$ & $0.42\pm0.02$ & $0.12\pm0.04$\\
$3\cdot10^{5}$ & $929$ & $0.0148$ & $0.41\pm0.02$ & $0.09\pm0.04$\\
\bottomrule
\end{tabular}
\end{table}

The test error stays between $0.37$ and $0.42$ and the mean LLR between $0.09$ and $0.18$ while $T$ grows by a factor of $30$, consistent with the prediction that the error remains bounded away from zero; for a fixed alternative the mean LLR would grow linearly in $T$. The test uses the conditional likelihood given the simulated past rather than the likelihood of the observed window, so it is not the Bayes test for the observation scheme of Theorem~\ref{thm:main}; its error is informative about detectability but is not bounded by $1/2$ by construction. Figure~\ref{fig:twopoint} summarises both experiments.

\begin{figure}
\centering
\includegraphics[width=\textwidth]{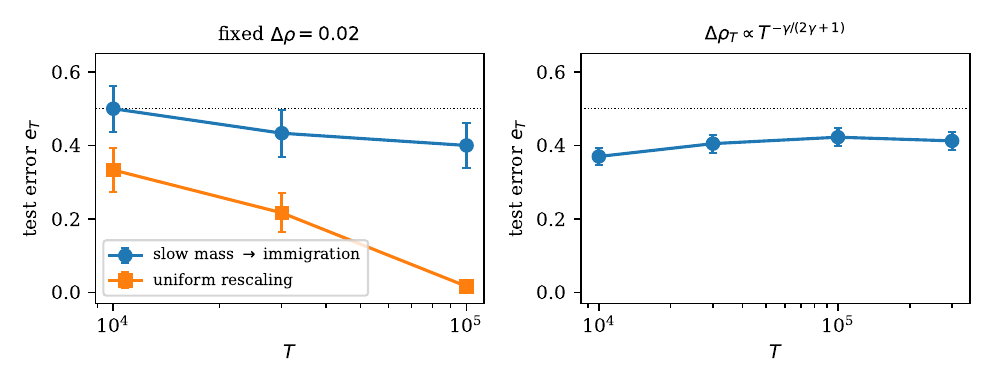}
\caption{Error of the zero-threshold LLR test. Left: fixed $\Delta\rho=0.02$, slow-mass alternative versus uniform rescaling. Right: $\Delta\rho_T\propto T^{-\gamma/(2\gamma+1)}$ as in the proof of Theorem~\ref{thm:main}. Bars are binomial standard errors.}\label{fig:twopoint}
\end{figure}

\subsection{Bias of flexible rational fits}

Table~\ref{tab:K} and Figure~\ref{fig:K} report maximum likelihood fits with $K$ exponentials on paths of length $T=2\cdot10^5$ generated after a burn-in of $2\times10^6$, compared with the quadratic-surrogate pseudo-truth. The four-parameter Omori fit recovers $\hat\rho=0.803$ and $0.790$ in two replicates, with a likelihood indistinguishable from that of $K=8$.

\begin{table}
\caption{Bias $\rho_0-\hat\rho$ of $K$-exponential fits ($\rho_0=0.8$, $T=2\cdot10^5$, 20 replicates) and the quadratic-surrogate pseudo-truth.}\label{tab:K}
\begin{tabular}{@{}rccc@{}}
\toprule
$K$ & Simulated bias & Surrogate & $\hat\mu$ (true $0.2$)\\
\midrule
1 & $0.186\pm0.001$ & $0.199$ & $0.385$\\
2 & $0.082\pm0.002$ & $0.087$ & $0.281$\\
3 & $0.044\pm0.002$ & $0.046$ & $0.243$\\
4 & $0.023\pm0.003$ & $0.026$ & $0.222$\\
6 & $0.019\pm0.003$ & flat & $0.218$\\
8 & $0.017\pm0.003$ & flat & $0.216$\\
\bottomrule
\end{tabular}
\end{table}

\begin{figure}
\centering
\includegraphics[width=\textwidth]{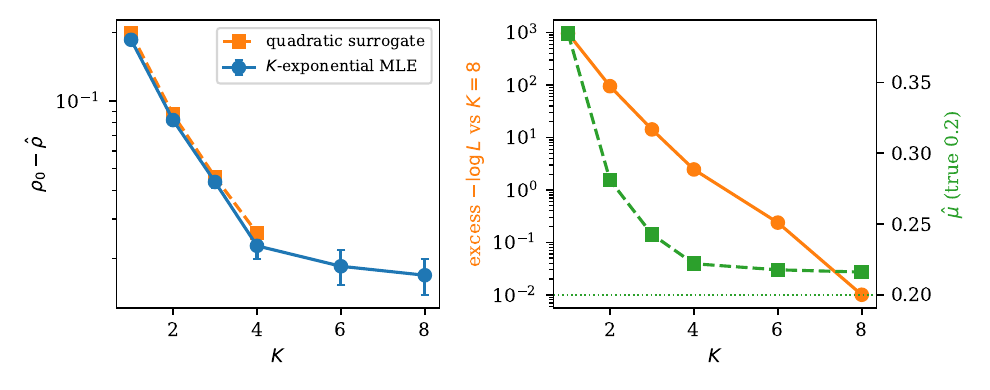}
\caption{Left: bias of $K$-exponential fits and the quadratic-surrogate pseudo-truth. Right: excess negative log-likelihood relative to $K=8$ and estimated baseline. The likelihood stops improving while the bias stalls and the baseline absorbs the missing branching mass.}\label{fig:K}
\end{figure}

For the best rational fit ($K\in\{6,8\}$), a joint regression of the log-bias over eight combinations of duration $T$ and immigration rate $\mu$ gives an exponent $-0.32\pm0.05$ in $T$, close to $-\gamma/(2\gamma+1)=-0.27$ (Figure~\ref{fig:duration}). A model based on duration and immigration rate fits the simulation grid substantially better than a model based on the number of events alone (weighted residual sums of squares $3.2$ and $75$ on $5$ and $6$ degrees of freedom; details in \citealp{HerreraMarinSupp}). At a comparable number of events, longer records give smaller bias.

\begin{figure}
\centering
\includegraphics[width=0.74\textwidth]{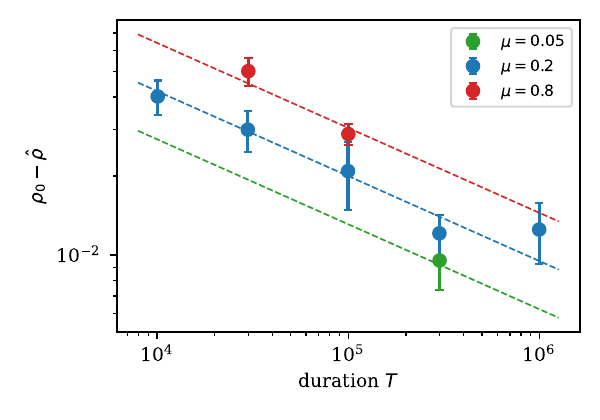}
\caption{Bias of the best rational fit ($K\in\{6,8\}$) against the observed duration, for three immigration rates. Dashed lines: joint power-law fit.}\label{fig:duration}
\end{figure}

\subsection{The block-dispersion estimator}

Table~\ref{tab:block} reports the mean absolute error of~\eqref{eq:blockest} over 40 paths of the Omori null ($\rho_0=0.8$, burn-in $8\times10^5$) for several block lengths. With $W\propto T^{1/(2\gamma+1)}$ the log-log slopes of the error in $T$ are $-0.25$, $-0.30$ and $-0.22$ for the three constants, close to $-\gamma/(2\gamma+1)=-0.27$. For the Omori experiment, a fixed block length produces a non-vanishing bias, of about $0.28$. At $T=10^6$, multiplying $W$ by four reduces the bias by a factor $2.33$, corresponding to an exponent $0.61$, in agreement with the $W^{-\gamma}$ bias of Lemma~\ref{lem:bias}. The estimator is computed exactly as in~\eqref{eq:blockest}, with $\mu_{\min}=0.05$, $\mu_{\max}=1$ and $\rho_{\max}=0.95$. Doubling or quadrupling the burn-in changes the slopes by at most $0.02$ \citep{HerreraMarinSupp}.

\begin{table}
\caption{Mean absolute error of the block-dispersion estimator ($\rho_0=0.8$, $\gamma=0.6$, 40 paths).}\label{tab:block}
\begin{tabular}{@{}lccc@{}}
\toprule
Block length & $T=10^4$ & $T=10^5$ & $T=10^6$\\
\midrule
$W=10$ & $0.282$ & $0.281$ & $0.280$\\
$W=T^{1/(2\gamma+1)}/4$ & $0.228$ & $0.140$ & $0.081$\\
$W=T^{1/(2\gamma+1)}$ & $0.118$ & $0.069$ & $0.037$\\
$W=4\,T^{1/(2\gamma+1)}$ & $0.063$ & $0.033$ & $0.016$\\
\bottomrule
\end{tabular}
\end{table}

\subsection{Confidence intervals}

Table~\ref{tab:ci} reports 300 paths of the Omori null ($\rho_0=0.8$, $\gamma=0.6$, burn-in $8\times10^5$) for the rate-optimal block length $W=T^{1/(2\gamma+1)}$ and for $W=T^{1/(2\gamma+1)}\log T$, with nominal level $0.95$. With the rate-optimal block length the bias is about twelve standard deviations and the interval never covers: bias and standard deviation have the same order but very different constants. With the longer blocks the predicted standard deviation $(1-\rho)\sqrt{W/(2T')}$ matches the empirical one (ratio $1.13$, $0.99$, $0.95$), the bias-to-standard-deviation ratio decreases as $(\log T)^{-(\gamma+1/2)}$, and coverage rises to $0.937$ at $T=10^7$. A Kolmogorov--Smirnov test does not reject normality of the standardised estimator at $T=10^7$ ($p=0.91$).

\begin{table}
\caption{Confidence intervals for $\rho$ (nominal level $0.95$, 300 paths). ``sd ratio'' is the empirical over the predicted standard deviation.}\label{tab:ci}
\begin{tabular}{@{}rcccccc@{}}
\toprule
 & \multicolumn{2}{c}{$W=T^{1/(2\gamma+1)}$} & \multicolumn{4}{c}{$W=T^{1/(2\gamma+1)}\log T$}\\
$T$ & bias/sd & coverage & $W$ & sd ratio & bias/sd & coverage\\
\midrule
$10^5$ & $-11.2$ & $0.00$ & $2157$ & $1.13$ & $-1.06$ & $0.890\pm0.018$\\
$10^6$ & $-11.8$ & $0.00$ & $7373$ & $0.99$ & $-0.86$ & $0.900\pm0.017$\\
$10^7$ & $-12.2$ & $0.00$ & $24498$ & $0.95$ & $-0.64$ & $0.937\pm0.014$\\
\bottomrule
\end{tabular}
\end{table}

\subsection{Adaptation}

Table~\ref{tab:lepski} compares a calibrated version of the adaptive estimator~\eqref{eq:lepski}, which does not use $\gamma$, with the block length $W=T^{1/(2\gamma+1)}$, which does, and with the best fixed grid block length chosen in hindsight, over 100 paths of Omori nulls with $\gamma\in\{0.4,0.6,0.8\}$. The calibrated rule differs from~\eqref{eq:lepski} in three respects: the threshold uses the plug-in standard deviation $\hat V_l\sqrt{2W_l/T'}$ instead of the class bound $\bar\sigma_l$, $x_T=\sqrt{\log T}$ (that is, $\varepsilon=0$), and the factor $6$ is replaced by $1$; the final estimate is clipped and projected as in~\eqref{eq:blockest}. The constants of~\eqref{eq:lepski} are chosen for the proof and are conservative in practice. Burn-in lengths are $3.2\times10^6$, $8\times10^5$ and $4\times10^5$ for $\gamma=0.4$, $0.6$ and $0.8$; doubling and quadrupling the burn-in at $\gamma=0.4$, $T=10^6$ changes the mean absolute errors by less than one Monte Carlo standard error \citep{HerreraMarinSupp}. The adaptive estimator matches or improves on the known-$\gamma$ choice for $\gamma=0.6$ and $0.8$ and is worse by a constant factor for $\gamma=0.4$. Its log-log slopes in $T$, computed from three values of $T$ and therefore only indicative, are $-0.27$, $-0.29$ and $-0.31$, against $-\gamma/(2\gamma+1)=-0.22$, $-0.27$ and $-0.31$. The hindsight oracle is about three times more accurate, which reflects the constants of the Lepski rule rather than its rate.

\begin{table}
\caption{Mean absolute error of the adaptive estimator, of the known-$\gamma$ block length and of the hindsight-optimal grid block length ($\rho_0=0.8$, 100 paths).}\label{tab:lepski}
\begin{tabular}{@{}cc ccc@{}}
\toprule
$\gamma$ & $T$ & adaptive & $W=T^{1/(2\gamma+1)}$ & hindsight oracle\\
\midrule
$0.4$ & $10^5$ & $0.1471$ & $0.0922$ & $0.0557$\\
$0.4$ & $10^6$ & $0.0809$ & $0.0552$ & $0.0316$\\
$0.4$ & $10^7$ & $0.0423$ & $0.0304$ & $0.0128$\\
\midrule
$0.6$ & $10^5$ & $0.0790$ & $0.0678$ & $0.0269$\\
$0.6$ & $10^6$ & $0.0394$ & $0.0385$ & $0.0133$\\
$0.6$ & $10^7$ & $0.0209$ & $0.0211$ & $0.0069$\\
\midrule
$0.8$ & $10^5$ & $0.0479$ & $0.0601$ & $0.0177$\\
$0.8$ & $10^6$ & $0.0251$ & $0.0329$ & $0.0080$\\
$0.8$ & $10^7$ & $0.0117$ & $0.0176$ & $0.0037$\\
\bottomrule
\end{tabular}
\end{table}

\section{Discussion}\label{sec:discussion}

\parhead{Structure, not architecture, restores identification.} The alternative in Theorem~\ref{thm:main} is not an Omori kernel. A correctly specified parametric subclass of fixed dimension excludes it and is not subject to the lower bound, consistently with the Omori fits of Section~\ref{sec:numerics}. Any procedure that learns a generic branching--relaxation measure in $\Cg$, however expressive, is subject to the bound. Parametrising the relaxation spectrum explicitly is useful because it makes such structural restrictions, and the zero-frequency mass, visible and controllable. We do not claim a parametric rate in the near-critical regime.

\parhead{Bias and coverage.} The honest intervals of Corollary~\ref{cor:ci} pay for bias-free inference with a factor $\ell_T^{1/2}$ in length and, as Table~\ref{tab:ci} shows, their coverage approaches the nominal level only at the slow rate at which $\ell_T$ grows. An alternative that avoids the factor is a bias-aware fixed-length interval in the spirit of \citet{ArmstrongKolesar2020}, which widens the critical value by the worst-case bias over $\Cg$; Lemma~\ref{lem:bias} provides an explicit bound, but its constant is conservative. Sharper bias control under second-order regular variation of $\nu$ near zero would allow a Richardson-type correction combining two block lengths.

\parhead{Likelihood-based estimators.} We expect sieve maximum likelihood with $K_T=\mathrm{polylog}(T)$ exponentials to attain the minimax rate up to logarithmic factors; positive exponential-sum approximation of power laws with accuracy $\varepsilon$ over a range of scales requires $K\sim\log(1/\varepsilon)\log(\tau_{\max}/\tau_{\min})$ terms \citep{BeylkinMonzon2010}, which suggests $K_T=O((\log T)^2)$. The numerics of Section~\ref{sec:numerics} are consistent with this, but asymptotic theory for Hawkes maximum likelihood with growing dimension is not available. Two further gaps are the power $(\log T)^{\varepsilon}$ between Theorems~\ref{thm:adapt} and~\ref{thm:price}, which comes from the use of Chebyshev's inequality and would close with exponential concentration of the dispersion statistic, and the sharp adaptive rate under a common loss, which would require matching upper and lower bounds for the same risk. Proposition~\ref{prop:noadaptci} rules out honest adaptation across the nested classes considered here; whether adaptation can be restored under additional tail regularity or self-similarity assumptions is open.

\parhead{Near-critical regime.} Theorem~\ref{thm:main} concerns fixed $\rho_0<1$. In the nearly unstable regime of \citet{JaissonRosenbaum2016}, $\mu_T\asymp T^{\alpha-1}\to0$, and the rescaled intensity converges to a rough square-root Volterra process \citep{ElEuchRosenbaum2019}, whose limit law has an atom at zero \citep{FriesenGerholdWiedermann2026}. The bound $\lambda_1\ge\mu_1$ used in Proposition~\ref{prop:rate} is then too crude. During low-intensity periods the baseline is exposed, and the relative baseline difference between the two hypotheses equals $m/(1-\rho)$, the relative error in the branching margin. Whether such periods occupy a non-vanishing fraction of a prelimit trajectory, and thereby reveal the branching margin that bursts conceal, is left open.

\parhead{Limitations.} The analysis is univariate and assumes a constant baseline. Non-stationary immigration biases branching-ratio estimates upwards \citep{FilimonovSornette2015,WehrliWheatleySornette2021}, whereas the mechanism studied here biases flexible estimators downwards. The two effects can partially compensate in empirical work; separating them requires joint modelling.

\begin{acks}[Acknowledgments]
AI-assisted tools (OpenAI ChatGPT and Anthropic Claude) were used for literature organisation, code assistance and editorial revision. The author independently verified all mathematical arguments, computations, references and reported results, and takes full responsibility for the manuscript.
\end{acks}

\begin{funding}
No specific funding was received for this work.
\end{funding}

\begin{supplement}
\stitle{Supplement to ``Minimax and adaptive inference for branching ratios under long memory''}
\sdescription{Simulation designs; quadrature and burn-in audits; details of the duration-versus-count regression; derivation of the quadratic surrogate; additional simulations for the block-dispersion estimator and for window-based fits. Available among the ancillary files of the arXiv version.}
\end{supplement}

\begin{supplement}
\stitle{Reproducibility package}
\sdescription{Python scripts and stored outputs that regenerate all tables and figures of the paper and the supplement (entry point \texttt{reproduce\_all.py}). Available among the ancillary files of the arXiv version.}
\end{supplement}

\bibliographystyle{spr-ims-nameyear}
\bibliography{refs}

\end{document}